\documentclass{article}

\usepackage{graphicx}%
\usepackage{multirow}%
\usepackage{amsmath,amssymb,amsfonts}%
\usepackage{amsthm}%
\usepackage{mathrsfs}%
\usepackage[title]{appendix}%
\usepackage{xcolor}%
\usepackage{textcomp}%
\usepackage{manyfoot}%
\usepackage{booktabs}%
\usepackage{algorithm}%
\usepackage{algorithmicx}%
\usepackage{algpseudocode}%
\usepackage{listings}%

\usepackage{amsfonts,amssymb}
\usepackage{url}
\usepackage{kbordermatrix}
\usepackage{caption,subcaption}
\usepackage{amsmath,amsfonts,graphicx,tikz,pgf,endnotes,multirow,blkarray,bigdelim,arydshln,scalefnt,verbatim}
\usepackage[dvipdfmx,dvisvgm]{animate}
\usetikzlibrary{decorations.text,shadings,shadows,shadows.blur,shapes.symbols}
\usetikzlibrary[patterns]

\usepackage[small,compact]{titlesec}
\usepackage[paper=a4paper,margin=1in]{geometry}
\usepackage{setspace}

\newtheorem{theorem}{Theorem}
\newtheorem{proposition}[theorem]{Proposition}%

\newtheorem{example}{Example}%
\newtheorem{remark}{Remark}%

\title{Projective rigidity of point-line configurations in the plane}
\author{Leah Wrenn Berman \and Signe Lundqvist \and Bernd Schulze \and Brigitte Servatius \and Herman Servatius \and Klara Stokes \and Walter Whiteley}
\date{}

\begin{document}
\maketitle
\abstract{In this paper, we develop tools for analyzing incidence-preserving  (or projective) motions of point-line configurations in the real projective plane. We define a real projective variety associated with a rank 3 matroid, that is closely related to the realization space of the matroid. Projective motions are described as curves within this variety, and we introduce the projective rigidity matrix to study these motions. The kernel of this matrix characterizes infinitesimal motions of a configuration, while its co-kernel identifies self-stresses. Whenever part of the matroid corresponds to the points and lines involved in certain projective theorems, such as Desargues' classical theorem, then the rank of the projective rigidity matrix is reduced, enabling the use of our rigidity matrix for the study of such theorems. 
Since symmetry often plays a significant role in the construction of point-line configurations, we also introduce a symmetry-adapted projective rigidity matrix for analyzing symmetry-preserving motions, including Euclidean, projective, and autopolarity-preserving motions. }

\maketitle

\section{Introduction}

The study of realization spaces of matroids is a rich and challenging topic, with deep connections to algebraic geometry, combinatorics, and geometric rigidity. For matroids of rank 3, realization spaces capture the possible geometric arrangements of points and lines in the projective plane consistent with the combinatorial structure of the matroid. These spaces have been a focal point in matroid theory, with applications to for example oriented matroids and problems in computational geometry~\cite{Oriented_matroids,Computational_geometry}. Realization spaces of matroids are also fundamental in optimization: matroid realization spaces are semialgebraic sets, and Mn\"ev's universality theorem states that any semialgebraic set defined over the integers is stably equivalent to the realization space of an oriented matroid~\cite{Mnev}. 
Understanding the realizability of matroids over a given field is a central question, yet it is highly nontrivial. Mnëv's universality theorem 
implies that realization spaces of matroids can exhibit arbitrarily complicated topological and geometric behavior. In a recent paper, it was shown that the matroid realization spaces of $\mathbb{C}$-realizable matroids of rank $3$ with a ground set of size 11 or less is smooth, and that for $n \geq 12$, there is a rank $3$ matroid with ground set of size $n$, such that its realization space over $\mathbb{C}$ has singularities \cite{Smooth_real_spaces}. 

From a geometric perspective, the study of real realizations of rank $3$ matroids can be naturally framed in terms of projective configurations of points and lines in the real projective plane. Projective configurations are classical objects in geometry, with roots in the work of T.~Kirkman, T.~Reye, E.~Steinitz, S.~Cohn-Vossen, and D.Hilbert \cite{hilcohn,pise,coxeter}. These configurations define incidence relations: a point is incident with a line if it lies on the line, and every pair of points spans a line while every pair of lines meets at a point. Such structures are preserved under projective transformations. While simple configurations are easily constructed, more intricate configurations, particularly those with higher incidence counts at each object, pose a significant challenge.

Projective configurations offer a unifying framework for bridging geometric and combinatorial approaches to rank $3$ matroids. In matroid terms, the flats correspond to lines, and the bases are triples of points not all collinear. The realization space of such a matroid can be studied through two algebraic frameworks: the \textit{matroid realization space} and the more recently introduced \textit{slack realization space}. 

A useful computational tool for working with matroid realization spaces is the new computer algebra system OSCAR \cite{OSCAR}. See for example \cite{matroidsoscar} for how OSCAR can be applied to matroid realization spaces.

In this paper, we extend the algebraic and geometric study of realization spaces by introducing new tools that enable the analysis of projective configurations of points and lines through the lens of rigidity theory. 
We define a real projective variety associated with a rank 3 matroid, that contains all points of the classical realization space of the matroid, but also points that correspond to degenerate realizations. Projective motions are described as curves within this variety, and we introduce the {\em projective rigidity matrix} to study these motions. The projective rigidity matrix is itself a real realization of the corresponding {\em projective rigidity matroid} with ground set the set of incidences between the points and the lines of the configuration. 
The projective rigidity matrix provides an efficient linear algebraic framework for studying projective configurations. Specifically:
\begin{itemize}
    \item \textbf{Kernel}: The kernel of the matrix corresponds to \textit{infinitesimal motions} of the configuration, which represent initial velocities of continuous, incidence-preserving deformations.
    \item \textbf{Co-kernel}: The co-kernel identifies \textit{self-stresses}, or row dependencies, which are algebraic signatures of implied incidences. For example, classical incidence theorems like Pappus and Desargues manifest as specific self-stresses.
\end{itemize}

Symmetry, a fundamental aspect of algebraic combinatorics, plays a prominent role in projective configurations. A \textit{symmetry} of a configuration is a projective transformation that preserves the incidence structure, while an \textit{auto-polarity} is a projective polarity that leaves the configuration invariant. Motivated by the question whether symmetry can be the underlying reason of existence of certain projective configurations, we introduce the \textit{projective orbit rigidity matrix}. We prove that the kernel of this matrix consists of the infinitesimal motions that preserve symmetries or auto-polarities, thereby providing a refined tool for studying symmetric configurations.

To illustrate the utility of these matrices, we apply them to various examples, demonstrating their ability to detect rigidity, flexibility, and symmetry-preserving motions in projective configurations. These tools also offer new insights into the realization spaces of matroids, particularly in the presence of symmetry.

The paper is organized as follows. Section~\ref{classical} introduces key concepts in projective geometry and projective configurations. Section~\ref{sec:rig} develops the notion of (continuous and infinitesimal) rigidity for projective configurations. Section~\ref{stress_section} explores self-stresses as row dependencies of the projective rigidity matrix, with connections to classical incidence theorems. Section~\ref{sec:sym} focuses on symmetric configurations and introduces the projective orbit rigidity matrix for studying symmetry-preserving motions. Finally, Section~\ref{sec:con} concludes with a discussion of future directions.

\section{Preliminaries}\label{classical}

Let $PG(\mathbb{R},2)$ be the real projective plane. Denote by $(x:y:z)$ and $(a:b:c)$ the homogeneous coordinates of a point and a line in $PG(\mathbb{R},2)$, respectively. We call points with $z=0$ points at infinity and lines with $c=0$ are called lines through the origin. 
A projective configuration  is a collection of points and lines in $PG(\mathbb{R},2)$, together with the incidence relation defined by symmetrized inclusion. 
This definition includes infinite configurations, but in this work we only consider configurations with a finite number of points and lines. 

Any projective configuration has an underlying combinatorial object $(P,L,I)$ consisting of the set of points $P$, the set of lines $L$ and the incidence relation $I$ between $P$ and $L$. This is an \emph{incidence geometry} of rank 2 since it has two types of objects. Such an incidence geometry $(P,L,I)$ can also be regarded as a bipartite graph with vertex set $P \cup L$ and edge set $I$, or as a hypergraph with vertex set $P$ and hyperedge set $L$. 

Also, given a rank $2$ incidence geometry, the triples of points that are not incident to a common line are the bases of a rank 3 matroid (the two notions of rank are unrelated). 

A \textit{flat} of a matroid with ground set $E$ is a subset $F \subseteq E$  of the ground set such that the rank of $F \cup \{e\}$ is larger than the rank of $F$ for any element $e \in E \setminus F$. One of many cryptomorphic ways to define a matroid is by its set of flats: a family $\mathcal{F}$ of subsets of $E$ is the set of flats of a matroid if  and only if
\begin{enumerate}
    \item $E \in \mathcal{F}$,
    \item If $F_1 \in \mathcal{F}$ and $F_2 \in \mathcal{F}$, then $F_1 \cap F_2 \in \mathcal{F}$, and
    \item If $F \in \mathcal{F}$ and $\{F_1,F_2,...,F_k\}$ is a minimal set of members of $\mathcal{F}$ that properly contain $F$, then the sets $F_1\setminus F$, $F_2\setminus F$,..., $F_k \setminus F$ partition $E \setminus F$. 
\end{enumerate}

The lines of an incidence geometry, and the pairs of points that are not incident to a common line of the incidence gemetry, are the flats of rank 2 of the rank 3 matroid defined by the incidence geometry. See for example one of the textbooks by Oxley or Bj\"orner et al. for more about matroids and matroid realizations \cite{Oriented_matroids,Oxley}.

From a geometric perspective, it is interesting to study the realization spaces of rank $3$ matroids. There are two algebraic models of the realization space.

First, we describe the matroid realization space. Let $M$ be a rank $3$ matroid. Assign the variables $p_i=(x_i,y_i,z_i)$ to the element $p_i$ of the ground set of $M$, and let $[p_ip_kp_j]$ denote the determinant of the $3 \times 3$-matrix with column vectors $p_i$, $p_j$ and $p_k$. The matroid realization space of $M$ is the subset of the Grassmannian of the lines of the projective plane defined by the following:
\[
  \begin{cases}
  [p_ip_kp_j]=0 & \quad \text{if } \{p_i, p_j, p_k \} \text{ is not a basis of } M\\
[p_ip_kp_j] \neq 0  & \quad \text{if } \{p_i, p_j, p_k \} \text{ is a basis of } M.
  \end{cases}
\]
This definition can be found for example in work by Bj\"orner et al. \cite{Oriented_matroids}.

In contrast, the slack realization space, recently introduced by Brandt and Wiebe \cite{Brandt}, is defined from the symbolic slack matrix $S_M(x)$; for a rank $3$ matroid $M$ this is a matrix with rows indexed by the points, and columns indexed by the lines. The entry $S_M(x)_{ij}$ is $0$ if the element $p_i$ lies on the line $l_j$, and a variable $x_{ij}$ if the point $p_i$ does not lie on the line $l_j$. Let $t$ be the number of variables in $S_M(x)$, and let $J$ be the ideal in $\mathbb{R}[x_1,...,x_t]$ defined by the $4 \times 4$-minors of the slack matrix. The slack ideal $I_M$ is the saturation of $J$ with respect to the principal ideal generated by the product of all variables $x_{ij}$. The slack realization space is the variety $V(I_M)$ in $\mathbb{R}^t$ defined by the slack ideal. The matrix $S_M(s)$ realizes the matroid $M$ if and only if $s \in I_M$ \cite{Brandt}.

In the literature, the term combinatorial \emph{$(p_r, l_k)$-configuration} is used to denote a purely combinatorial incidence geometry of rank 2 with the following  properties: (i) there are $p$ ``points" and $l$ ``lines",  (ii) there are positive integers $r$ and $k$ such that $r$ ``lines" go through each ``point", and $k$ ``points" are on each ``line", and (iii) each pair of ``lines" meet in at most one ``point" and each pair of ``points" span at most one ``line". Note that the number of incidences in this case is $|I|=pr=lk$. In the language of hypergraphs, a combinatorial $(p_r, l_k)$-configuration is an $r$-regular and $k$-uniform hypergraph with the property that no two hyperedges share more than one vertex.
If $r=k$, then the configuration is called {\em balanced} (sometimes \emph{symmetric}) and is referred to as a combinatorial \emph{$v_k$ configuration} if it has $v$ points.

The literature of projective configurations is often concerned with questions of existence and constructions. Some combinatorial incidence geometries cannot be realized as points and straight lines in the projective plane in such a way that all points and lines are distinct. Given an incidence geometry, one can therefore ask whether it is realizable. In this paper, we start with a realization of a combinatorial incidence geometry, and investigate the projective rigidity properties of that realization.

 A (projective planar) {\em realization} of an incidence geometry $S=(P,L,I)$ is an assignment of a point $\mathbf{p}_j=(x_j:y_j:z_j)$ in the projective plane to each element $p_j$ of $P$, and an assignment of a line $\mathbf{l}_i=(a_i:b_i:c_i)$ in the projective plane to each element $l_i$ of $L$, such that $\mathbf{l}_i \cdot \mathbf{p}_j=0$ whenever $p_j$ and $l_i$ are incident. The equation $\mathbf{l}_i \cdot \mathbf{p}_j=0$ then encodes that the point $\mathbf{p}_j$ lies on the line $\mathbf{l}_i$. Note that this definition does not exclude that points or lines coincide. 

In Section \ref{sec:sym}, we will be particularly interested in symmetric realizations, which are realizations that are invariant under the action of a symmetry group. The symmetries that we are considering can be either Euclidean symmetries, for example reflections or rotations, or they can be autopolarities.

\section{Rigidity and flexibility of projective configurations}\label{sec:rig}

\label{proj_inf_motions}
 The \textit{realization space} $V(S)$ of an incidence geometry $S=(P,L,I)$ is the space of all collections of assignments $(\mathbf{p}, \mathbf{l})$ of points and lines in the projective plane such that $(S, \mathbf{p}, \mathbf{l})$ is a configuration of points and lines realizing $S$. In other words, the realization space is the real projective variety defined by the $|I|$ polynomials of the form $\mathbf{l}_i \cdot \mathbf{p}_j=0$.
Let $V_{\mathbf{p}, \mathbf{l}}(S)$ be the subvariety of $V(S)$ consisting of the realizations of $S$ that can be obtained from $(S, \mathbf{p}, \mathbf{l})$ by a projective transformation.

Suppose that a rank 3 matroid has as its bases the triples of three points that are not incident to a common line in an incidence geometry. Then the points of the Grassmannian of that matroid, and the points of its slack realization space can be seen as realizations of the incidence geometry as points and lines in $PG(\mathbb{R},2)^2$. However, the realizations in the Grassmannian of that rank $3$ matroid and in the slack realization space are exactly those where a realized point lies on a realized line if, and \textit{only} if, the combinatorial point is incident to the combinatorial line.  

The realization space defined above contains all collections of points and lines that satisfy the combinatorial incidences. Therefore, it can happen that some combinatorial lines or combinatorial points are coincident in the realization, i.e. $\mathbf{l}_j=\mathbf{l}_k$ for some pair of distinct combinatorial lines $l_j$ and $l_k$, or $\mathbf{p}_i=\mathbf{p}_m$ for some pair of distinct combinatorial points $p_i$ and $p_m$. 

We say that a configuration $(S,\mathbf{p}, \mathbf{l})$ is \textit{(projectively) rigid} if all configurations $(S, \mathbf{p'}, \mathbf{l'})$ in a neighborhood of $(S,\mathbf{p}, \mathbf{l})$ in the realization space are in $V_{\mathbf{p}, \mathbf{l}}(S)$. A configuration that is not rigid is said to be \textit{(projectively) flexible}.

There are several other possible definitions of rigidity for projective configurations of points and lines. One natural definition is as follows: A configuration $(S, \mathbf{p}, \mathbf{l})$ is \textit{(projectively) rigid} if all continuous paths $(\mathbf{p}(t), \mathbf{l}(t))$ in the realization space such that $\mathbf{p}(0)=\mathbf{p}$ and $\mathbf{l}(0)=\mathbf{l}$ remain in $V_{\mathbf{p}, \mathbf{l}}(S)$. With this definition, a configuration $(S, \mathbf{p}, \mathbf{l})$ is \textit{flexible} if there is a continuous path $(\mathbf{p}(t),\mathbf{l}(t))$ in $V(S)$ such that $\mathbf{p}(0)=\mathbf{p}$, $\mathbf{l}(0)=\mathbf{l}$, and such that there is some $t \in (0,1]$ for which $(S, \mathbf{p}(t), \mathbf{l}(t))$ is not in $V_{\mathbf{p}, \mathbf{l}}(S)$.

Instead of assuming that the paths in the above definitions are continuous, we can assume that the paths are smooth, or we can assume that they are analytic. The next theorem says that all these possible definitions are equivalent to our original definition.

\begin{theorem}
\label{def_equivalence}
    Let $(S, \mathbf{p}, \mathbf{l})$ be a projective configuration. Then the following are equivalent:

    \begin{enumerate}
        \item[(a)] $(S, \mathbf{p}, \mathbf{l})$ is flexible.
        \item[(b)] There is a continuous path $(\mathbf{p}(t),\mathbf{l}(t))$ in $V(S)$ such that $\mathbf{p}(0)=\mathbf{p}$, $\mathbf{l}(0)=\mathbf{l}$, and such that there is some $t \in (0,1]$ for which $(S, \mathbf{p}(t), \mathbf{l}(t))$ is not in $V_{\mathbf{p}, \mathbf{l}}(S)$.
        \item[(c)] There is a smooth path $(\mathbf{p}(t),\mathbf{l}(t))$ in $V(S)$ such that such that $\mathbf{p}(0)=\mathbf{p}$, $\mathbf{l}(0)=\mathbf{l}$, and such that there is some $t \in (0,1]$ for which $(S, \mathbf{p}(t), \mathbf{l}(t))$ is not in $V_{\mathbf{p}, \mathbf{l}}(S)$.
        \item[(d)] There is an analytic path $(\mathbf{p}(t),\mathbf{l}(t))$ in $V(S)$ such that $\mathbf{p}(0)=\mathbf{p}$, $\mathbf{l}(0)=\mathbf{l}$, and such that there is some $t \in (0,1]$ for which $(S, \mathbf{p}(t), \mathbf{l}(t))$ is not in $V_{\mathbf{p}, \mathbf{l}}(S)$.
    \end{enumerate}
\end{theorem}

\begin{proof}
    Clearly $(d)$ implies $(c)$ and $(c)$ implies $(b)$. 

    Now, if $(b)$ holds, then there is some largest $t_0$ so that  $(\mathbf{p}(t_0), \mathbf{l}(t_0))$ is in $V_{\mathbf{p}, \mathbf{l}}(S)$. By definition, there is a projective transformation $A$ that takes $(S, \mathbf{p'}, \mathbf{l'})$ to $(S, \mathbf{p}, \mathbf{l})$. Every neighborhood of $(S, \mathbf{p}, \mathbf{l})$ intersects the curve $(S, A\mathbf{p}(t), A\mathbf{l}(t))$, for $t \in (t_0, 1]$, and every such intersection necessarily contains configurations not in $V_{\mathbf{p},\mathbf{l}}(S)$, by definition of $t_0$. It follows that $(S, \mathbf{p}, \mathbf{l})$ is flexible.

    It remains to show that $(a)$ implies $(d)$. If $(S, \mathbf{p}, \mathbf{l})$ is not rigid, then there is some realization $(S, \mathbf{p'}, \mathbf{l'})$ in $U \cap (V(S) \setminus V_{\mathbf{p},\mathbf{l}}(S)$, where $U$ is some neighborhood of $(S, \mathbf{p}, \mathbf{l})$ in $V(S)$. Without loss of generality, we can assume that both realizations have no points at infinity or lines that go through the origin.

    The realizations that have only finite points, and only lines that do not go through the origin correspond to points on the real affine variety defined by a polynomial of the form
\begin{equation}
    a_i x_j + b_i y_j + 1=0
\end{equation}
for each incidence $(p_j, l_i)$.
    
    By [\cite{alg_approx_of_curves}, Lemma 18.3], there are analytic curves $\mathbf{p}(t)$ and $\mathbf{l}(t)$ in this real affine variety such that $\mathbf{p}(0)=\mathbf{p}$ and $\mathbf{p}(1)=\mathbf{p'}$, and $\mathbf{l}(0)=\mathbf{l}$ and $\mathbf{l}(1)=\mathbf{l'}$. The curves $\mathbf{p}(t)$ and $\mathbf{l}(t)$ lift to analyic curves in $V(S)$ such that $z(t)$ and $c(t)$ are constant for all points and lines respectively.
    Since $(S, \mathbf{p'}, \mathbf{l'})$ cannot be obtained from $(S, \mathbf{p}, \mathbf{l})$ by a projective transformation, it follows that $(a)$ implies $(d)$.
\end{proof}

\begin{remark}
    The motions of interest in this paper are the motions that preserve incidences of points and lines modulo projective transformations. {\em Pinning} is a technique that is useful to illustrate the congruence classes of motions. When we pin a configuration we fix four points in general position, that is, in a position such that no three of the points are collinear. Pinning a configuration will leave only the motions that we are interested in.
    We refer the reader to~\cite{fields} for more details.
\end{remark}

The geometric realization of a finite line $l_i$ that does not go through the origin is given by its homogeneous coordinates
$\mathbf{l}_i=(a_i\colon b_i\colon 1)$, and the realization of a finite point $p_j$ is given by $\mathbf{p}_j=(x_j\colon y_j\colon 1)$. The constraint equation 
\begin{equation}
    a_i x_j + b_i y_j + 1=0
    \label{incidence_eq}
\end{equation}
expresses that the point $p_j$ and the line $l_i$ are incident. Finding a configuration of points and lines realizing a given incidence geometry amounts to solving $|I|$ equations of this form.

Suppose we have a path in the realization space, and so regard the vectors $\mathbf{p_j}$ and $\mathbf{l_i}$ as differentiable functions of parameter $t$, say time, whose initial position
corresponds to an initial configuration. 
Since the constraint equations must be satisfied for all values of $t$, the Jacobian of Equation (\ref{incidence_eq}) gives a linear equation
\begin{equation}
    \mathbf{p}_j\cdot \Delta \mathbf{l}_i + \mathbf{l}_i\cdot \Delta \mathbf{p}_j=0
    \label{Jacobian_eq}
\end{equation}
for each incidence $(p_j, l_i) \in I$. The \textit{projective rigidity matrix} is the coefficient matrix $M(S, \mathbf{p}, \mathbf{l})$ of the system of equations of the form (\ref{Jacobian_eq}). The row of the projective rigidity matrix that corresponds to the incidence $(p_j, l_i)$ looks as follows:
\[
\left[\begin{array}{ c ccccc ccccc }
  &0&\dots  & x_j \ y_j &\dots&0&\ldots& a_i\  b_i &\ldots&0& 
\end{array}\right],
\]

A \emph{(projective) infinitesimal motion} of a configuration $(S,\mathbf{p}, \mathbf{l})$ is an element of the kernel of $M(S, \mathbf{p}, \mathbf{l})$. Linearizations of projective transformations are infinitesimal motions of all configurations of points and lines.  We say that a configuration is \textit{(projectively) infinitesimally rigid} if all its infinitesimal motions are linearizations of projective transformations. A configuration that is not infinitesimally rigid is \textit{(projectively) infinitesimally flexible}. The following relation between rigidity and infinitesimal rigidity holds:

\begin{theorem}[\cite{fields}]
    If a configuration is projectively infinitesimally rigid, then it is projectively rigid.
    \label{inf_rigid_implies}
\end{theorem}

As the kernel of $M(S, \mathbf{p}, \mathbf{l})$ is at least 8-dimensional, it follows that an infinitesimally rigid configuration must satisfy $|I|\geq 2|L|+2|P|-8$. A configuration that is \emph{minimally infinitesimally rigid} in the sense that removing any incidence makes the configuration infinitesimally flexible, must satisfy 
\begin{eqnarray}
    |I| = 2|L|+2|P|-8 \text{ and} \label{min_rig}\\
    |I'| \leq 2|L(I')|+2|P(I')|-8,
    \label{count_indep}
\end{eqnarray}
where the inequality must hold for all non-empty subsets of incidences $I' \subseteq I$. A configuration such that the rows of $M(S, \mathbf{p}, \mathbf{l})$ are linearly independent has to satisfy (\ref{count_indep}). Such configurations are said to be \emph{independent}.

\begin{example}[The Desargues configuration]
\label{Desargues_ex}
    Desargues's theorem states that two triangles are perspective from a point if and only if they are perspective from a line. 
    The Desargues configuration, see Figure \ref{Desargues}, occurs as a consequence of Desargues's theorem. The points of the Desargues configuration are the vertices of two triangles, $a$, $b$ and $c$ and $a'$, $b'$ and $c'$ respectively, as well as the a point of perspective $p$ of the two triangles and the points of intersection of the lines $ab$ and $a'b'$, $ac$ and $a'c'$ and $bc$ and $b'c'$. The points of intersection will be collinear as a consequence of Desargues's theorem. 

    The Desargues configuration has $10$ points, $10$ lines and $30$ incidences. So, $|I|=2|P|+2|L|-10$, which means that the Desargues configuration is expected to have a two-dimensional space of motions, excluding the projective transformations.

    However, suppose that the points $a$, $b$, $c$ and $p$ are fixed. Then the points $a'$, $b'$ and $c'$ can be moved independently along the lines $ap$, $bp$ and $cp$ respectively, and the three remaining points will stay collinear as a consequence of Desargues's theorem. Hence the Desargues configuration has a three-dimensional space of motions, rather than the expected two-dimensional space of motions.
\end{example}

\begin{figure}
    \centering
    \begin{tikzpicture}
    \filldraw[black] (1/4,1/4) circle (2.4pt);
    \filldraw[black] (-1/4,7/4) circle (2.4pt);
    \filldraw[black] (-3/4,13/4) circle (2.4pt);
    \filldraw[black] (1/4,5/4) circle (2.4pt);
    \filldraw[black] (1/4,11/4) circle (2.4pt);
    \filldraw[black] (4/4,7/4) circle (2.4pt);
    \filldraw[black] (6/4,11/4) circle (2.4pt);
    \filldraw[black] (24/4,7/4) circle (2.4pt);
    \filldraw[black] (10/4,11/4) circle (2.4pt);
    \filldraw[black] (-11/4,17/4) circle (2.4pt);

    \draw (1/4,1/4) -- (-3/4,13/4);
    \draw (1/4,1/4) -- (1/4,11/4);
    \draw (1/4,1/4) -- (6/4,11/4);
    \draw (1/4,5/4) -- (-11/4,17/4);
    \draw (1/4,11/4) -- (-11/4,17/4);
    \draw (-11/4,17/4) -- (24/4,7/4);
    \draw (-3/4,13/4) -- (24/4,7/4);
    \draw (-1/4,7/4) -- (24/4,7/4);
    \draw (1/4,5/4) -- (10/4,11/4);
    \draw (1/4,11/4) -- (10/4,11/4);

    \node[anchor=north] at (1/4,1/4) {$p$};
    \node[anchor=west] at (1/4,5/4) {$a'$};
    \node[anchor=north] at (0.42,11/4) {$a$};
    \node[anchor=north] at (-3/4,13/4) {$b$};
    \node[anchor=north] at (6/4,11/4) {$c$};
    \node[anchor=north] at (-1/4,7/4) {$b'$};
    \node[anchor=north] at (4/4,7/4) {$c'$};
    
\end{tikzpicture}
    \caption{The $10_3$ Desargues configuration.}
    \label{Desargues}
\end{figure}
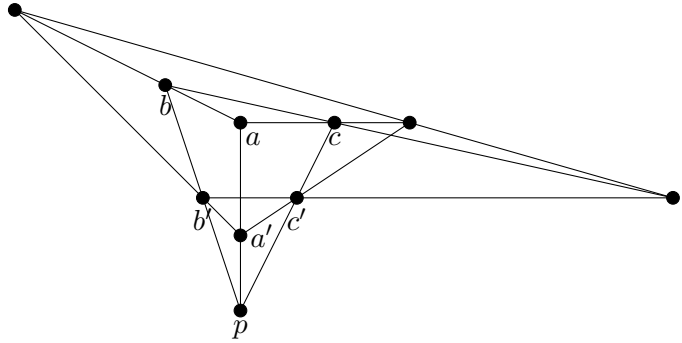

The Desargues configuration is an example where counting the points, lines and incidences does not predict the correct number of motions, which means that the natural sparsity counts of Equations~(\ref{min_rig}) and (\ref{count_indep}) are not sufficient for characterizing independence in the rigidity matrix and minimal rigidity. 

Note that it is a theorem in projective geometry that causes the discrepancy between the sparsity counts and the rank of the rigidity matrix. In the Desargues configuration, one of the incidences is implied by Desargues's theorem, which causes the Desargues configuration to be more flexible than expected.

We conjecture that this is part of a more general phenomenon, namely that for projective theorems regarding point-line configurations, where certain incidences imply others, these statements can be formulated in terms of the projective rigidity matrix having a non-trivial row dependence. More evidence for this conjecture can be found in \cite{fields}. Row dependencies, or self-stresses, are the subject of Section \ref{stress_section} in this paper.

We conclude this section with an example of an infinitesimally rigid  configuration.

\begin{figure}[htp]
\centering
\includegraphics[scale=1.2]{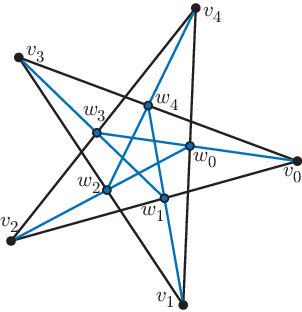}
\caption{A cyclic $10_3$-configuration.}
\label{Mr10}
\end{figure}

\begin{example}
\label{Rigid_ex}
A combinatorial $v_k$-configuration  is \textit{cyclic} if there is a cyclic permutation of the points that sends the configuration to itself. The configuration shown in Figure~\ref{Mr10} is a cyclic $10_3$ configuration with $10$ points and $10$ lines (non-isomorphic to the Desargues configuration).  It can be drawn with 5-fold rotational symmetry and two symmetry classes of points and lines. Geometric configurations with this property are sometimes called  \emph{astral configurations} \cite[Chapter 3]{Gru2009b}.  
This particular configuration can be constructed as follows \cite{BerDeOFau2020}:
\begin{enumerate}
\item Construct points $v_0, \ldots, v_4$ as the vertices of a regular $5$-gon centered at $\mathcal{O}$. Typically we choose $\mathcal{O} = (0,0)$, and let $v_i = (\cos(2 \pi i/5), \sin(2 \pi i/5))$.
\item Construct lines $l_i = v_iv_{i+2}$.
\item Construct a circle $\mathcal{C}$ passing through the points $v_1$, $\mathcal{O}$, and $v_{-1}$, and let $w_0$ be one of the two intersections of $\mathcal{C}$ and $l_0$. 
\item Construct $w_i$ as the rotation of $w_0$ by $\frac{2\pi i}{5}$ about $\mathcal{O}$. 
\item Construct $m_i = w_iw_{i+2}$.
\end{enumerate}
The Configuration Construction Lemma (see \cite{BerDeOFau2020}) can be used to show that since $w_0$ was constructed as the intersection of that particular circle with that particular line, each line $m_i$ passes through the point $v_{i+2}$.

Now the $20_4$-configuration in Figure \ref{Mr20} can be constructed from the $10_3$-configuration in Figure \ref{Mr10} as follows:

    \begin{enumerate}
        \item Construct a circle $\mathcal{C}'$ passing through the points $v_1$, $\mathcal{O}$, and $v_{3}$, and let $w_0'$ be one of the two intersections of $\mathcal{C}$ and $l_0$.
        \item Construct $w_i'$ as the rotation of $w_0'$ by $\frac{2\pi i}{5}$ about $\mathcal{O}$.
        \item Construct lines $n_i=w_i'w_{i+2}'$, which will pass through $v_{i+3}$.
        \item Construct the points $u_i$ as the intersections of the lines $m_i$ and $n_i$. 
        \item Construct the lines $o_i=u_iu_{i+1}$, which will pass through the points $w_{i+1}'$ and $w_{i+3}$.
    \end{enumerate}

    The $20_4$ configuration constructed this way is projectively infinitesimally rigid (as can be checked by computing the rank of the projective rigidity matrix), and therefore rigid, by Theorem \ref{inf_rigid_implies}.  
\end{example}

\begin{figure}[htb]
\centering
\includegraphics[scale=1.1]{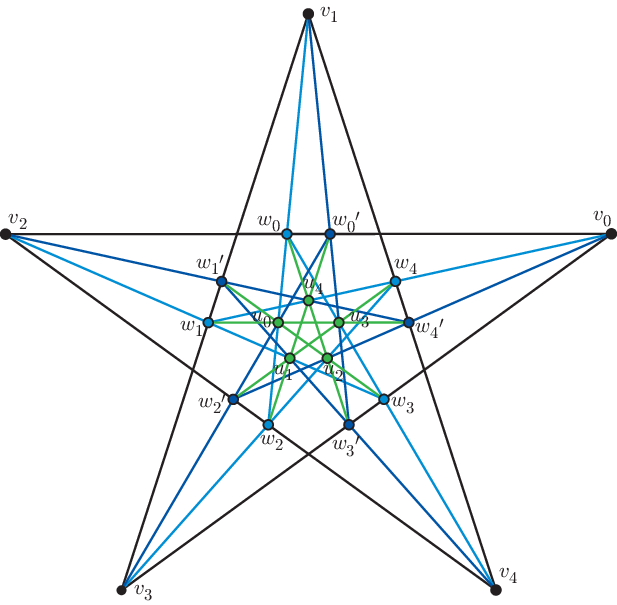}
\caption{A cyclic $20_4$-configuration.}
\label{Mr20}
\end{figure}

\section{Self-stresses of projective configurations}
\label{stress_section}

In this section, we are going to introduce the theory of row dependencies, or self-stresses. The theory of row dependencies is a dual theory of column dependencies (infinitesimal motions).

Recall that the vectors $\mathbf{p}_i$ and $\mathbf{l}_j$ are the homogeneous coordinates assigned to the point $p_i$ and $l_j$ in a configuration realizing $S$. In this section, we want to consider $\mathbf{p}_i$ and $\mathbf{l}_j$ as vectors in $\mathbb{R}^3$. In order to do this, we pick the representatives of $\mathbf{p}_i$ and $\mathbf{l}_j$ such that the last non-zero coordinate is 1. In this section, we do not assume that the points are finite and that the lines do not go through the origin. However, for finite points and lines that do not go through the origin, we pick the same representatives that we used to define the rigidity matrix.

A \textit{stress} of a realization of an incidence geometry $S=(P,L,I)$ is a scalar $\omega_{i,j}$ assigned to each incidence $(p_i,l_j) \in I$. For technical  reasons we define the stress to be zero for all non-incident pairs. We say that a stress is an \emph{equilibrium stress} (or \emph{self-stress}) if the following equations hold for each point $p_i \in P$ and line $l_j \in L$:

\begin{equation}
\sum_{l_j \in L}\omega_{i,j}\mathbf{l}_j = \mathbf{0},
\qquad
\sum_{p_i \in P}\omega_{i,j}\mathbf{p}_i = \mathbf{0}.
\label{stress_equations}
\end{equation}

Note that as the scalar $\omega_{i,j}$ is defined to be zero for all non-incident pairs of points and lines, the summands of the sum $\sum_{l_j \in L}\omega_{i,j}\mathbf{l}_j$ are non-zero only for the lines $l_j$ such that $(p_i, l_j) \in I$. Similarly, the summands in the sum $\sum_{p_i \in P}\omega_{i,j}\mathbf{p}_i = \mathbf{0}$ are non-zero only for the points $p_i$ such that $(p_i, l_j) \in I$.

Note that the existence of an equilibrium stress does not depend on the chosen representatives of $\mathbf{p}_i$ and $\mathbf{l}_j$. However, the coefficients $\omega_{i,j}$ do, so for the coefficients to be well-defined we choose to fix specific representatives of $\mathbf{p}_i$ and $\mathbf{l}_j$.

As the next proposition shows, existence of an equilibrium stress is also invariant under projective transformations.

\begin{proposition}
    \label{proj_eq_stress}
    If a configuration realizing an incidence geometry $S=(P,L,I)$ has an equilibrium stress, then any projectively equivalent configuration realizing $S$ also has an equilibrium stress. 
\end{proposition}

\begin{proof}
    Suppose that $(S, \mathbf{p}, \mathbf{l})$ is a configuration that has an equilibrium stress, where $\omega_{i,j}$ is the stress coefficient assigned to the incidences $(p_i,l_j) \in I$. Let $(S, \mathbf{p'}, \mathbf{l'})$ be the realization of $S$ obtained from $(S, \mathbf{p}, \mathbf{l})$ by a projective transformation represented by the matrix $A$. 

    Note that $A\mathbf{p}_i=\lambda_i\mathbf{p}'_i$, where $\mathbf{p}'_i$ is the representative of the homogeneous coordinates of $\mathbf{p}'_i$ such that the last non-zero coordinate is 1. Similarly, $A\mathbf{l}_j=\lambda_j\mathbf{l}'_j$. 

    Define new stress coefficients by $\omega_{i,j}'=\lambda_i \lambda_j \omega_{i,j}$ for each incidence $(p_i,l_j)$. Now consider the sum $\Sigma_{l_j \in L} \omega'_{i,j} \mathbf{l}'_j$. By definition, $$\Sigma_{l_j \in L} \omega'_{i,j} \mathbf{l}'_j=\Sigma_{l_j \in L} \lambda_i \lambda_j \omega_{i,j} \mathbf{l}'_j$$ and by assumption, $$\Sigma_{l_j \in L} \lambda_i \lambda_j \omega_{i,j} \mathbf{l}'_j= \lambda_i A( \Sigma_{l_j \in L}  \omega_{i,j} \mathbf{l}_j)= \mathbf{0}.$$ Similarly, $$\Sigma_{p_i \in P} \omega'_{i,j} \mathbf{p}'_i=\lambda_j A( \Sigma_{p_i \in P}  \omega_{i,j} \mathbf{p}_i)= \mathbf{0}.$$ 
    Hence $\omega'_{i,j}$ is an equilibrium stress on the realization $(S, \mathbf{p'}, \mathbf{l'})$.
\end{proof}

Suppose that all points are finite, so that $\mathbf{p}_i=(x_i:y_i:1)$ for all points $p_i \in P$, and no lines go through the origin, so that $\mathbf{l}_j=(a_j:b_j:1)$ for all lines $l_j \in L$. Under these assumptions, the rigidity matrix is well-defined, and the existence of an equilibrium stress clearly implies a row-dependence in the matrix $M(S,\mathbf{p}, \mathbf{l})$. 

On the other hand, if there is a row-dependence in the rigidity matrix, then there are scalars $\omega_{i,j}$ so that the equations 

\begin{equation}
\sum_{l_j \in L}\omega_{i,j}(a_j,b_j) = \mathbf{0},
\qquad
\sum_{p_i \in P}\omega_{i,j}(x_i,y_i) = \mathbf{0}
\label{stress2}
\end{equation}
hold for all points $p_i \in P$ and lines $l_j \in L$. Equation (\ref{stress_equations}) is satisfied for a point $p_i$ if Equation (\ref{stress2}) is satisfied, and the sum $\Sigma_{l_j: (p_i,l_j) \in I}w_{i,j}$, which gives the third coordinate of Equation (\ref{stress_equations}), is zero. Dually, for lines, Equation (\ref{stress2}) has to hold, and the sum  $\Sigma_{p_i: (p_i,l_j) \in I}w_{i,j}$ has to be 0. The equations $\sum_{l_j \in L: (p_i, l_j) \in I} \omega_{i,j} = 0$ and $\sum_{p_i \in P: (p_i, l_j) \in I} \omega_{i,j} = 0$ say that,
for a configuration placed with no points at infinity or lines through the origin, every self-stress $\omega$ must be purely combinatorial, that is,
it must be a row dependence of the combinatorial incidence matrix.

Now, note that $(a_j,b_j) \cdot (x_i,y_i)=-1$ for all pairs $(p_i,l_j) \in I$. It follows that 
$$(\sum_{l_j \in L}\omega_{i,j}(a_j,b_j)) \cdot (x_i,y_i) = - \Sigma_{l_j : (p_i,l_j) \in I} \omega_{i,j} = \mathbf{0}.$$
The dual statement holds for lines. Hence, a row dependence in the rigidity matrix implies that the configuration has an equilibrium stress. 

The row space of the matrix is the set of vectors of the form
$$\sum_{i,j}\omega_{i,j} Row_{i,j},$$
for any scalars $\omega_{i,j}$.

The  conditions that the row space of $M(S,\mathbf{p}, \mathbf{l})$ has rank
$2|P| + 2|L| - 8$, i.e. the configuration is infinitesimally rigid,
is equivalent to the condition of the configuration being {\em statically rigid}, which says that
the vectors of the from
$\omega M(S,\mathbf{p}, \mathbf{l})$ span the entire orthogonal complement of the
space of trivial motions.

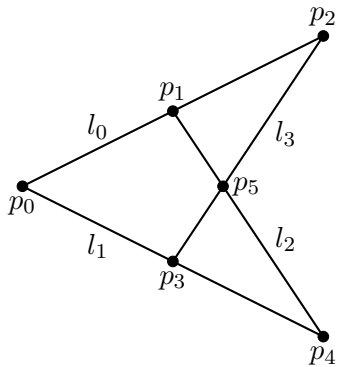
\begin{figure}
\begin{center}
\begin{tikzpicture}
\filldraw[black] (0,0) circle (2pt);
\filldraw[black] (2,1) circle (2pt);
\filldraw[black] (4,2) circle (2pt);
\filldraw[black] (2,-1) circle (2pt);
\filldraw[black] (4,-2) circle (2pt);
\filldraw[black] (8/3,0) circle (2pt);

\draw[thick] (0,0)--(4,2);
\draw[thick] (0,0)--(4,-2);
\draw[thick] (2,1)--(4,-2);
\draw[thick] (2,-1)--(4,2);

\node[anchor=north] at (0,0) {$p_0$};
\node[anchor=south] at (2,1) {$p_1$};
\node[anchor=south] at (4,2) {$p_2$};
\node[anchor=north] at (2,-1) {$p_3$};
\node[anchor=north] at (4,-2) {$p_4$};
\node[anchor=west] at (8/3,0) {$p_5$};
\node[anchor=south] at (1,1/2) {$l_0$};
\node[anchor=north] at (1,-1/2) {$l_1$};
\node[anchor=north] at (3.5,1) {$l_3$};
\node[anchor=south] at (3.5,-1) {$l_2$};
\end{tikzpicture}
\end{center}
\caption{The complete quadrilateral.}
\label{quadrilateral}
\end{figure}

\begin{example}[The complete quadrilateral]
\label{quadrilateral_ex}
    A complete quadrilateral is a configuration with $4$ lines and their $6$ intersection points, see Figure \ref{quadrilateral}. We want to find homogeneous coordinates for the points and lines of a self-stressed complete quadrilateral.

    Every point of the complete quadrilateral is incident to exactly two lines. If the point $p_i$ lies on the lines $l_j$ and $l_k$, then $\omega_{i,j}\mathbf{l}_j+\omega_{i,k}\mathbf{l}_k=0$. As we are interested in the case where $\omega_{j,i}$ and $\omega_{i,k}$ are non-zero, we must have that $\mathbf{l}_j$ is a scalar multiple of $\mathbf{l}_k$. Geometrically, this means that $\mathbf{l}_j$ and $\mathbf{l}_k$ are the same line. Because the complete quadrilateral is connected, this implies that all lines have to be the same in a self-stressed realization. 
    
    Furthermore, if $l_j$ and $l_k$ are incident to the point $p_i$, we get that $\omega_{i,j}=-\omega_{i,k}$. 

    In this case, this means that $\omega_{0,0}=-\omega_{0,1}$, $\omega_{1,0}=-\omega_{1,2}$, $\omega_{2,0}=-\omega_{2,3}$, $\omega_{3,1}=-\omega_{3,3}$, $\omega_{4,1}=-\omega_{4,2}$ and $\omega_{5,2}=-\omega_{5,3}$. 

    Now, if we restrict to configurations such that the points are finite and the lines do not go through the origin, then we can use that $\Sigma_{l_j: (p_i,l_j) \in I}w_{i,j}=0$ for all lines $l_j$. Together with the above observation that $\omega_{i,j}=-\omega_{i,k}$ for $p_i$ incident to $l_j$ and $l_k$, we get the following four equations
    \begin{eqnarray*}
        \omega_{0,0}+\omega_{1,0}+\omega_{2,0}=0 \\
        -\omega_{0,0}+\omega_{3,1}+\omega_{4,1}=0 \\
        -\omega_{1,0} -\omega_{4,1} + \omega_{5,2} =0 \\
        \omega_{2,0}+\omega_{3,1}+\omega_{5,2}=0
    \end{eqnarray*}
    where the equations correspond to $l_0$, $l_1$, $l_2$ and $l_3$ respectively. It is easily verified that one solution to this system of equations is $\omega_{0,0}=1$, $\omega_{1,0}=1$, $\omega_{2,0}=-2$, $\omega_{3,1}=-1$, $\omega_{4,1}=2$ and $\omega_{5,2}=3$.

    To find a self-stressed realization of the complete quadrilateral with the above coefficients, we need a solution to the following four equations corresponding to the lines 

    \begin{eqnarray*}
        \mathbf{p}_0+\mathbf{p}_1-2\mathbf{p}_2=0 \\
        -\mathbf{p}_0-\mathbf{p}_3+2\mathbf{p}_4=0 \\
        -\mathbf{p}_1 -2\mathbf{p}_4 + 3\mathbf{p}_5 =0 \\
        -2\mathbf{p}_2-\mathbf{p}_3+3\mathbf{p}_5=0
    \end{eqnarray*}
    such that the points $\mathbf{p}_0$, $\mathbf{p}_1$, $\mathbf{p}_2$, $\mathbf{p}_3$, $\mathbf{p}_4$ and $\mathbf{p}_5$ all lie on a line. For simplicity, we pick the points such that they all lie on the line $y=1$. We can pick the points $\mathbf{p}_5$, $\mathbf{p}_4$ and $\mathbf{p}_3$ arbitrarily on the line. The coordinates of the points $\mathbf{p}_2$, $\mathbf{p}_1$ and $\mathbf{p}_0$ can be computed from the first three points. One example of points that satisfy the system of equations is $\mathbf{p}_0=(-2:1:1)$, $\mathbf{p}_1=(8:1:1)$, $\mathbf{p}_2=(3:1:1)$, $\mathbf{p}_3=(0:1:1)$, $\mathbf{p}_4=(-1:1:1)$ and $\mathbf{p}_5=(2:1:1)$. With these coordinates, the complete quadrilateral has a single self-stress.
\end{example}

A \textit{weaving} of lines in the plane is a directed graph $G=(V,E)$ together with an assignment of a line in the projective plane that does not go through the origin to each vertex, such that the lines $\mathbf{l}_i$ and $\mathbf{l}_j$ meet in a finite point whenever $(i,j) \in E$. Let $(x_{ij},y_{ij},1)$ be the intersection point of $\mathbf{l}_i$ and $\mathbf{l}_j$. A \textit{self-stress} of a weaving is an assignment of a scalar $s_{ij}$ to each edge $(i,j) \in E$, such that $s_{ij}=-s_{ji}$ and

\begin{equation}
\sum_{j: (i,j) \in E}s_{ij}(x_{ij},y_{ij},1) = \mathbf{0}.
\label{stress_weaving}
\end{equation}

In Example \ref{quadrilateral_ex}, we find that a dependence in the projective matrix includes a check that the scalars, restricted to the columns for the lines, form a self-stress for the weaving of the lines. See~\cite{Weaving,Weaving1}. The converse does not hold.

\section{Symmetric projective configurations}\label{sec:sym}

Many examples of projective configurations exhibit symmetry. There are several reasons for this; symmetry implies beauty, and usually the symmetry makes it easier to construct the configuration. 
For example, the $21_4$ Gr\"unbaum-Rigby configuration, whose description in 1990 \cite{GruRig1990} began the modern study of configurations, is typically drawn with 7-fold dihedral symmetry. Many papers  have constructed interesting examples of configurations (for example, \cite{BobPis2003,BerDeOFau2020,berman,berman2,berman3,BerFau2013}, and \cite{Gru2009b} has many other examples), including those shown in Figures \ref{Mr10} and \ref{Mr20},  by leveraging symmetry and geometry to prove that the necessary incidences occur. 

In this section, we will consider the effect of symmetry on projective rigidity. We will develop a projective orbit rigidity matrix, which is analogous to the orbit rigidity matrix introduced by B.Schulze and W.Whiteley for studying forced symmetric infinitesimal rigidity of bar-joint frameworks in Euclidean space  \cite{SW11}. In this context, the projective orbit rigidity matrix will allow us to study the space of realizations of an incidence geometry, subject to a given symmetry.

\subsection{Symmetries, dualities and polarities of projective space}

The projective general group $PGL(3,\mathbb{R})$ is the group of real invertible $3\times 3$-matrices modulo multiplication with a scalar, and by the fundamental theorem of projective geometry it acts upon the real projective plane as its group of symmetries (its collineation group), because $\mathbb{R}$ has only trivial field automorphisms. 

Any non-degenerate quadratic form $Q$ on $\mathbb{R}^3$ defines a polarity $\pi_Q:\mathbb{R}^3\rightarrow(\mathbb{R}^3)^*$ that sends $v$ to $vQ$.
This induces a geometric polarity exchanging the points with the lines of the real projective plane. 
The image of a projective point $p$ under a polarity is the polar line of $p$ for the conic defined by the quadratic form. 
The identity matrix defines a quadratic form corresponding to the purely imaginary conic $x^2+y^2+z^2=0$, and the polarity it defines maps a point to the line with the same homogeneous coordinates as the point. 

The correlation group of the real projective plane is the group $PGL(3,\mathbb{R})$ extended with a polarity. Given one polarity, the other polarities can be obtained by composing with a projective transformation. For example, the polarity defined by the identity matrix $Q_1=\textup{Id}$, corresponding to the imaginary conic, can be composed with the projective transformation with matrix representative in $GL(3,\mathbb{R})$ equal to $$T=\begin{pmatrix}1&0&0\\0&1&0\\0&0&-1\end{pmatrix}$$ to obtain the polarity defined by the matrix $$Q_2=\begin{pmatrix}1&0&0\\0&1&0\\0&0&-1\end{pmatrix}$$ corresponding to the conic with equation $x^2+y^2=z^2$, because $Q_2=Q_1T$. 

The orthogonal group $O(3)<PGL(3,\mathbb{R})$ acts upon $\mathbb{R}^3$. The special orthogonal group $SO(3)$ is the subgroup of $O(3)$ consisting of the direct orthogonal isometries.  
The projective orthogonal group $PO(3)$ and the projective special orthogonal group $PSO(3)$ are the corresponding subgroups of $PGL(3,\mathbb{R})$ and describe the induced actions upon the real projective plane. In odd dimension $n$, $PO(n)=PSO(n)\cong SO(n)$. For $n=3$ one can find a very nice description of the finite subgroups of this group in \cite{Octon}. 

A planar finite projective configuration consists of a finite subset of points and lines of the projective plane. The symmetry group of the configuration is a finite subgroup of $PGL(3,\mathbb{R})$ that preserves the point set, the line set, and the incidences. 
A polarity of the configuration is an involutory correlation that preserves the incidences of the configuration but permutes the point and the line sets. A configuration that has a non-trivial polarity is called autopolar. See Figure~\ref{autopolar} for an example of an autopolar configuration. 

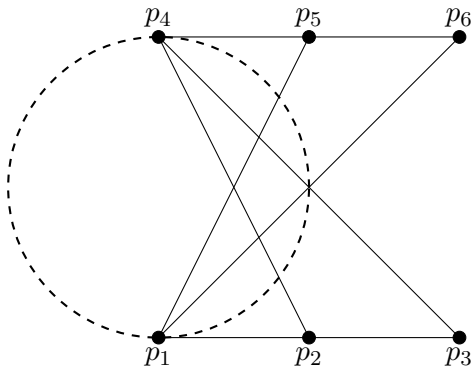
\begin{figure}[htp]
    \centering
    \begin{tikzpicture}[x=2cm,y=2cm]
    \filldraw[black] (0,-1) circle (2.4pt);
    \filldraw[black] (0,1) circle (2.4pt);
    \filldraw[black] (1,-1) circle (2.4pt);
    \filldraw[black] (2,-1) circle (2.4pt);
    \filldraw[black] (1,1) circle (2.4pt);
    \filldraw[black] (2,1) circle (2.4pt);
    \draw[thick, dashed] (0,0) circle (1);
    \draw (0,-1)--(1,1);
    \draw (0,-1)--(2,1);
    \draw (0,1)--(1,-1);
    \draw (0,1)--(2,-1);
    \draw (0,-1)--(2,-1);
    \draw (0,1)--(2,1);

    \node[anchor=south] at (0,1) {$p_4$};
    \node[anchor=south] at (1,1) {$p_5$};
    \node[anchor=south] at (2,1) {$p_6$};
    \node[anchor=north] at (0,-1) {$p_1$};
    \node[anchor=north] at (1,-1) {$p_2$};
    \node[anchor=north] at (2,-1) {$p_3$};
    \end{tikzpicture}
    \caption{An autopolar configuration: the polarity $\pi_C$ sends the configuration  to itself, i.e. it sends the points of the configuration to the lines of the configuration and vice versa.}
    \label{autopolar}
\end{figure}
The group consisting of projective correlations (symmetries and polarities) preserving the configuration is called the correlation group of the configuration. 

\subsection{Projective orbit rigidity matrix}

In this section we will establish a symmetry-adapted rigidity matrix, called the projective orbit rigidity matrix, whose kernel consists of the infinitesimal motions that exhibit the same symmetry as the configuration. 
Hence the orbit rigidity matrix can be used to show the existence of symmetry-preserving deformations of symmetric projective configurations. 

Throughout this section, we will let $\Gamma$ denote a subgroup of the group of correlations of a projective configuration such that each element of $\Gamma$ is either an element of $PO(3)$ or a polarity defined by an orthogonal matrix.
We say that an infinitesimal projective motion of the  configuration is ``$\Gamma$-symmetric'' if it also has symmetry $\Gamma$. 

More formally, let $m$ be an element of the kernel of the projective rigidity matrix and let $\gamma$ be an element of $\Gamma$. The vector $m$ has two entries for each point of the configuration, and two entries for each line of the configuration. Let $m(p)$ denote the vector consisting of the two elements corresponding to the point $p$. We can then require that $m(\gamma p)=\gamma m(p)$ for all points $p$ of the configuration.

Similarly, we require that $ m(\gamma l)=\gamma m(l)$ for all lines of the configuration, where $m(l)$ denotes the vector consisting of the two entries of $m$ corresponding to a line $l$ of the configuration. An element $m$ of the kernel of the rigidity matrix satisfying these conditions for all points and lines and all elements of  $\Gamma$ is said to be \textit{$\Gamma$-symmetric}.

\subsubsection{Free actions}\label{subsec:freeaction}
Suppose that we have a configuration of points and lines with symmetry group $\Gamma$. For simplicity, we will first assume that no points or lines are fixed by non-trivial elements of $\Gamma$. In this case, the projective orbit rigidity matrix takes on a particularly simple form.

Consider an incidence $(q, \gamma r)$ of the configuration, where $q$ and $r$ are representatives of orbits of points and lines under the action of $\Gamma$, respectively.  Since we assumed that $m$ is $\Gamma$-symmetric, we have that $m(\gamma(r))=\gamma m(r)$, and the equation in the rigidity matrix corresponding to the incidence 
$(q, \gamma r)$ is
\begin{equation}
    q \cdot \gamma m(r) + \gamma r \cdot m(q)=0.
    \label{eq_1}
 \end{equation}
    Since the inner product is invariant under the action of $\Gamma$, (\ref{eq_1}) is equivalent to
\begin{equation}
    \gamma^{-1}q \cdot m(r) + \gamma r \cdot m(q)=0.
    \label{sym_eq}
\end{equation}
or, in matrix notation, to
\begin{equation}
    (\gamma^{-1}q)^T  m(r) + (\gamma r)^T m(q)=0.
    \label{sym_eq1}
\end{equation}
    For an incidence geometry $S$ and a configuration of points and lines $(S, \mathbf{p}, \mathbf{l})$ realizing $S$, let $M^\Gamma(S,\mathbf{p}, \mathbf{l})$ be the coefficient matrix of the system of equations obtained by considering one equation of the form (\ref{sym_eq}) for each orbit of incidences.
    The row of $M^\Gamma(S,\mathbf{p}, \mathbf{l})$ corresponding to an incidence $(q, \gamma r)$, where $q \neq r$, looks as follows:
    \[
\begin{array}{|c|c|c|c|c|}
    \hline
   & r& &q &\\
  \hline
  0 \dots 0 & (\gamma^{-1}q)^T & 0 \dots 0 & (\gamma r)^T & 0\\
  \hline
 \end{array}
 \]
    A row of $M^\Gamma(S,\mathbf{p}, \mathbf{l})$ corresponding to an incidence $(q, \gamma q)$ looks as follows:

    \[
\begin{array}{|c|c|c|}
    \hline
   & q& \\
  \hline
  0 \dots 0 & (\gamma q + \gamma^{-1}q)^T& 0 \dots 0\\
  \hline
 \end{array}
 \]
 Note that such an incidence can occur if $\gamma$ is a polarity, because in that case,  the image of a point under $\gamma$ is a line.

We say that $\bar m \in \mathbb{R}^{(2|P|+2|L|)/|\Gamma|}$ is the \emph{restriction} of a 
$\Gamma$-symmetric infinitesimal projective motion $m$ if $m(\gamma q)=\gamma \bar m(q)$ for all $\gamma\in \Gamma$, and all $q\in P\cup L$. The elements of the kernel of $M^\Gamma(S,\mathbf{p}, \mathbf{l})$ are in one-to-one correspondence with the 
$\Gamma$-symmetric infinitesimal projective motions, as the following result shows.

    \begin{theorem}
    Let $(S, \mathbf{p}, \mathbf{l})$ be a configuration of points and lines with symmetry group $\Gamma$. Then $m$ is an element of the kernel of $M^\Gamma(S,\mathbf{p}, \mathbf{l})$ if and only if it is the restriction of a $\Gamma$-symmetric infinitesimal projective motion of the configuration.
    \end{theorem}
    \begin{proof}
   Equation~(\ref{sym_eq}) is equivalent to all the equations coming from the incidences in the same orbit as 
   $(q, \gamma r)$.
   That is, if $(\beta q, \beta \gamma r)$ is such an incidence, then the equation coming from the incidence $(\beta q, \beta \gamma r)$ is
    \begin{equation}
    \begin{split}
    \beta q \cdot \beta \gamma \bar m(r) + \beta \gamma r \cdot \beta \bar m(q)=0.
    \end{split}
    \label{M_rep}
    \end{equation}
    which is equivalent to Equation~(\ref{sym_eq}).
    If $m$ is a $\Gamma$-symmetric infinitesimal  motion, then all equations of the form (\ref{M_rep}) are satisfied, so the restriction of $m$ satisfies Equation~(\ref{sym_eq}).
    Similarly, if $\bar m$ is an element of the kernel of $M^{\Gamma}_2(S,\mathbf{p})$, then the vector $m$ defined by $m(\gamma q)=\gamma \bar m(q)$ satisfies all equations of the form (\ref{eq_1}). Hence $m$ is a $\Gamma$-symmetric infinitesimal  motion with restriction $\bar m$.
    \end{proof}

\subsubsection{Actions with fixed points, lines, and incidences}

    Suppose that we have a symmetric configuration of points and lines with symmetry group $\Gamma$ such that the action of $\Gamma$ has some  point or line that is fixed by a non-trivial element of $\Gamma$. 
    
    To set up the correct system of equations for non-trivial elements $\gamma\in \Gamma$ that fix points, we note that if there is a point $q$ that is fixed by $\gamma$, then in any symmetry-preserving motion of the configuration the point $q$ should remain in the subspace of the projective plane that is fixed by the action of $\gamma$.

Let $F_{\gamma}$ be the subspace of the projective plane that is fixed by the action of  $\gamma \in \Gamma$. For a point $q$ of the configuration, we define the subspace
$$
U_q=\bigcap_{\gamma \in \Gamma, \gamma  q = q}F_\gamma.
$$
Note that if $q$ is only fixed by the identity, then $U_q$ is all of the projective plane. For each point $q$ of the configuration, pick a basis $B$ of $U_{q}$. Let $M_{q}$ be the matrix with columns given by $B$.

Similarly, we can define the subspace
$$
U_l=\bigcap_{\gamma \in \Gamma, \gamma l= l}F_\gamma
$$
for each line $l$ of the configuration.
For each line $l$ of the configuration, let $M_l$ be a basis matrix of $U_l$.

In the following, we let $m(r)=M_r\hat{m}(r)$ for a $\textrm{dim}(U_r)$-dimensional column vector $\hat m(r)$.
Similar to the free action case, for each incidence $(q, \gamma r)$ of the configuration, we consider the equation 
\begin{equation}q \cdot \gamma M_r \hat m(r) + \gamma r \cdot M_q \hat m(q)=0 \label{sym_fix0}\end{equation}
which, by the orthogonality of $\gamma$, is equivalent to 
\begin{equation} \gamma^{-1} q\cdot  M_r\hat m(r) + \gamma r \cdot M_q \hat m(q)=0 \label{sym_fix}\end{equation}
or, in matrix notation, to
\begin{equation} (\gamma^{-1} q)^T M_r\hat m(r) + (\gamma r)^T M_q \hat m(q)=0 \label{sym_fix2}\end{equation}
Let $M^\Gamma(S,\mathbf{p}, \mathbf{l})$ be the coefficient matrix of the system of equations where each incidence orbit is represented by an equation of the form (\ref{sym_fix}). The matrix $M^\Gamma(S,\mathbf{p}, \mathbf{l})$ has a row for each orbit of incidences under the action of $\Gamma$, and $\dim(U_r)$ columns for each orbit of points and lines $r$ under the action of $\Gamma$. Explicitly, the row corresponding to the incidence orbit represented by the incidence $(q,\gamma r)$ has the form:
  \[
\begin{array}{|c|c|c|c|c|}
    \hline
   & r& &q &\\
  \hline
  0 \dots 0 & (\gamma^{-1}q)^TM_r & 0 \dots 0 & (\gamma r)^TM_q & 0\\
  \hline
 \end{array}
 \]
Note that a point or line cannot be fixed by a polarity $\gamma$, and hence $r=q$ implies that $M_q$ is the identity matrix. So then we obtain the row for $(q,\gamma q)$ given in Section~\ref{subsec:freeaction}.

 We will see that the kernel of $M^\Gamma(S,\mathbf{p}, \mathbf{l})$ consists of the $\Gamma$-symmetric infinitesimal projective motions.

\begin{theorem}
    Let $(S, \mathbf{p}, \mathbf{l})$ be a $\Gamma$-symmetric configuration of points and lines. Then $\hat m$ is an element of the kernel of $M^\Gamma(S,\mathbf{p}, \mathbf{l})$ if and only if $\bar m$ defined by $\bar m(r)=M_r\hat{m}(r)$ for each representative $r$ of the point and line orbits under the action of $\Gamma$ is the restriction of a $\Gamma$-symmetric infinitesimal projective motion of $S$. 
\end{theorem}

\begin{proof}
    Suppose that $(q,\gamma r) \in I$ is a representative of an orbit of incidences.
    The equation in $M^\Gamma(S,\mathbf{p}, \mathbf{l})$ corresponding to $(q,\gamma r)$ is Equation~(\ref{sym_fix}).

Since the action of $\Gamma$ preserves inner products, Equation~(\ref{sym_fix}) is equivalent to
Equation~(\ref{sym_fix0}), which is the equation corresponding to the incidence $(q, \gamma r)$ in $M(S, \mathbf{p}, \mathbf{l})$.

Consider another incidence $(\beta q, \beta \gamma r)$ in the same orbit. Since the inner product is invariant under the action of $\Gamma$, Equation~(\ref{sym_fix0}) is also equivalent to
    \begin{equation}
    \begin{split}
    \beta q \cdot \beta \gamma M_r \hat m(r) + \beta \gamma r \cdot \beta M_q \hat m(q) =0.
        \end{split}
    \label{M_rep_eq}
\end{equation}

Equation (\ref{M_rep_eq}) is the equation corresponding to the incidence 
$(\beta q, \beta \gamma r)$ in $M(S, \mathbf{p}, \mathbf{l})$. Since Equation~(\ref{sym_fix0}) and Equation~(\ref{M_rep_eq}) are both equivalent to Equation~(\ref{sym_fix}), the theorem follows, since Equation~(\ref{sym_fix}) is satisfied if and only if $\hat m$ is an element of the kernel of $M^\Gamma(S,\mathbf{p},\mathbf{l})$, and Equation~(\ref{sym_fix0}) and Equation~(\ref{M_rep_eq}) are the equations that need to be satisfied for $\bar m$ to be the restriction of a $\Gamma$-symmetric projective motion.
\end{proof}

We conclude this section with two examples. 

First, we use the orbit rigidity matrix to analyze the configuration in Figure \ref{Symmetric_w._motions}. It has 13 points, 12 lines and 42 incidences. So its rigidity matrix has $2\times 13+2\times 12=50$ columns and $42$ rows, and hence the configuration would be projectively rigid if the rows were independent.  
The configuration in Figure \ref{Symmetric_w._motions} however, has two non-trivial  infinitesimal motions. 

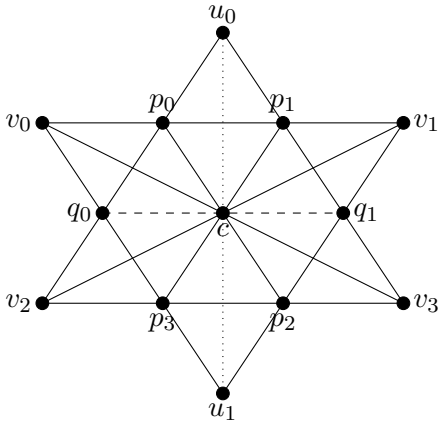
\begin{figure}[htp]
    \centering
    \begin{tikzpicture}[x=1.2cm,y=1.2cm]
    \filldraw[black] (-4,-5) circle (2.4pt);
    \filldraw[black] (-2,-4) circle (2.4pt);
    \filldraw[black] (-6,-4) circle (2.4pt);
    \filldraw[black] (-4,-1) circle (2.4pt);
    \filldraw[black] (-2,-2) circle (2.4pt);
    \filldraw[black] (-6,-2) circle (2.4pt);
    \filldraw[black] (-16/3,-3) circle (2.4pt);
    \filldraw[black] (-4,-3) circle (2.4pt);
    \filldraw[black] (-8/3,-3) circle (2.4pt);
    \filldraw[black] (-14/3,-2) circle (2.4pt);
    \filldraw[black] (-14/3,-4) circle (2.4pt);
    \filldraw[black] (-10/3,-4) circle (2.4pt);
    \filldraw[black] (-10/3,-2) circle (2.4pt);

    \draw (-6,-4)--(-4,-1);
    \draw (-6,-4)--(-2,-2);
    \draw (-4,-5)--(-6,-2);
    \draw (-4,-5)--(-2,-2);
    \draw (-2,-4)--(-4,-1);
    \draw (-2,-4)--(-6,-2);
    \draw[dashed] (-16/3,-3)--(-8/3,-3);
    \draw (-2,-4)--(-6,-4);
    \draw (-2,-2)--(-6,-2);
    \draw (-14/3,-4)--(-10/3,-2);
    \draw (-10/3,-4)--(-14/3,-2);
    \draw[dotted] (-4,-1)--(-4,-5);

    \node[anchor=south] at (-14/3,-2) {$p_0$};
    \node[anchor=south] at (-10/3,-2) {$p_1$};
    \node[anchor=north] at (-10/3,-4) {$p_2$};
    \node[anchor=north] at (-14/3,-4) {$p_3$};
    \node[anchor=east] at (-16/3,-3) {$q_0$};
    \node[anchor=west] at (-8/3,-3) {$q_1$};
    \node[anchor=east] at (-6,-2) {$v_0$};
    \node[anchor=west] at (-2,-2) {$v_1$};
    \node[anchor=east] at (-6,-4) {$v_2$};
    \node[anchor=west] at (-2,-4) {$v_3$};
    \node[anchor=south] at (-4,-1) {$u_0$};
     \node[anchor=north] at (-4,-5) {$u_1$};
     \node[anchor=north] at (-4,-3) {$c$};
\end{tikzpicture}
    \caption{A  configuration with dihedral symmetry $D_4$.}
    \label{Symmetric_w._motions}
\end{figure}

The configuration can be constructed with dihedral symmetry $D_4$ of order 4 as in Figure \ref{Symmetric_w._motions} using reflections in the two perpendicular lines (dashed and dotted) by first picking two points $q_0$ and $q_1$ on the dashed line. The dotted line is perpendicular to the dashed line, and passes through the midpoint of  $q_0$ and $q_1$.
Now, pick a point $p_0$ not on the symmetry lines, and reflect it in the dashed and dotted lines to get the points $p_1$, $p_2$ and $p_3$. The points $u_0$, $u_1$, $v_0$, $v_1$, $v_2$ and $v_3$ are the points of intersection of the lines $q_0p_0$ and $q_1p_1$, $q_0p_3$ and $q_1p_2$, $q_0p_3$ and $p_0p_1$, $q_1p_2$ and $p_0p_1$, $p_0q_0$ and $p_2p_3$ and $p_1q_1$ and $p_2p_3$ respectively. The point $c$ is the center of rotation.

There are some projective theorems involved in this configuration. Firstly, the line incident to $q_0$, $c$ and $q_1$ is the Pascal line of the hexagon given by the lines $v_0u_1$, $u_1v_1$, $v_1v_2$, $v_2u_0$, $u_0v_3$ and $v_0v_3$. By Pascal's theorem, a Pascal line exists if the points $v_0$, $u_0$, $v_1$, $v_3$, $u_1$ and $v_2$ lie on a conic. Secondly, by Brianchon's theorem the diagonals of the hexagon with edges $u_0v_1$, $v_1v_3$, $v_3u_1$, $u_1v_2$, $v_2v_2$ and $v_0u_0$ meet at a point if and only if the edges of the hexagon are tangent to a conic. Also by Brianchon's theorem, the diagonals of the hexagon with the edges $p_0p_1$, $p_1q_1$, $q_1p_2$, $p_2p_3$, $p_3q_0$ and $q_0p_0$ meet at a point if and only if the edges of the hexagon are tangent to a conic.

The condition of Pascal's theorem and the conditions of the two instances of Brianchon's theorem are necessarily satisfied in any realization of the configuration. Since the configuration can be constructed with $D_4$ symmetry, these symmetries must imply that the conditions of the theorems are satisfied in the $D_4$-symmetric configuration. However, there are realizations of the configuration in Figure \ref{Symmetric_w._motions} that  have neither of the reflectional symmetries nor the half-turn symmetry. In general, it is always possible to find projectively equivalent realizations of a symmetric configuration without Euclidean symmetries by applying a trivial projective motion that does not preserve the symmetry. However, for this example, we will see that every realization obtained by applying a non-trivial projective motion to the $D_4$-symmetric configuration will be projectively equivalent to a realization with reflectional symmetry.

The configuration has a two-dimensional space of non-trivial infinitesimal motions. Using the orbit rigidity matrix, one can show that there is a six-dimensional space of reflection-symmetric infinitesimal motions with respect to the dashed line. One can also show that there is a six-dimensional space of reflection-symmetric infinitesimal motions with respect to the dotted line. In both cases, there is a four-dimensional space of trivial infinitesimal motions, so for each of the two reflections, there is a two-dimensional space of non-trivial reflection-symmetric infinitesimal motions with respect to that reflection.

We can also set up an orbit rigidity matrix to see that there is a three-dimensional space of $D_4$-symmetric infinitesimal motions. The space of trivial  $D_4$-symmetric infinitesimal motions is two-dimensional, leaving one non-trivial $D_4$-symmetric motion. 

    Consequently, for each choice of reflection, there must be one non-trivial reflection-symmetric infinitesimal motion for that reflection, but not the other. 
    It can be verified that these infinitesimal motions extend to finite continuous motions, and hence for each reflection symmetry (in the dashed or dotted line) there is a realization of the configuration in Figure \ref{Symmetric_w._motions} that has this symmetry, which is not projectively equivalent to the configuration in Figure \ref{Symmetric_w._motions}.  
    However, the configurations with a single reflectional symmetry obtained in this way will be projectively equivalent to each other.

    As all non-trivial projective motions preserve one of the reflectional symmetries, all configurations obtained from the configuration in Figure \ref{Symmetric_w._motions} by applying a non-trivial motion will be projectively equivalent to a realization with reflectional symmetry. It is therefore not clear that there \textit{is} a realization of the incidence structure in Figure \ref{Symmetric_w._motions} that is not projectively equivalent to a reflection-symmetric configuration. 
    
    Is there, for example, a non-trivial realization of the incidence structure in Figure \ref{Symmetric_w._motions} such that the rigidity matrix has full rank equal to $42$? Or, is there  a realization of the same incidence structure with only one non-trivial infinitesimal motion? Such realizations would not be projectively equivalent to the realization in Figure \ref{Symmetric_w._motions}.

\begin{figure}[htp]
    \centering
    \begin{tikzpicture}[x=2cm,y=2cm]
    \filldraw[black] (0,-1) circle (2.4pt);
    \filldraw[black] (0,1) circle (2.4pt);
    \filldraw[black] (1,-1) circle (2.4pt);
    \filldraw[black] (2,1) circle (2.4pt);
    \filldraw[black] (3,-1) circle (2.4pt);
    \filldraw[black] (2/3,1) circle (2.4pt);
    \draw[thick, dashed] (0,0) circle (1);
    \draw (0,-1)--(2,1);
    \draw (0,1)--(1,-1);
    \draw (0,-1)--(2/3,1);
    \draw (0,1)--(3,-1);
    \draw (0,-1)--(2,-1);
    \draw (0,1)--(2,1);
    \draw[thick, red] (2,-1)--(3,-1);
    \draw[->,thick, red] (2,-1)--(2.925,-1);
    \draw[thick, red] (1,1)--(2/3,1);
    \draw[->,thick, red] (1,1)--(9/12,1);

    \node[anchor=south] at (0,1) {$p_4$};
    \node[anchor=south] at (1,1) {$p_5$};
    \node[anchor=south] at (2,1) {$p_6$};
    \node[anchor=north] at (0,-1) {$p_1$};
    \node[anchor=north] at (1,-1) {$p_2$};
    \node[anchor=north] at (2,-1) {$p_3$};
    \end{tikzpicture}
    \caption{A non-trivial motion of the configuration in Figure~\ref{autopolar} preserving autopolarity.}
    \label{autopolar_motion}
\end{figure}
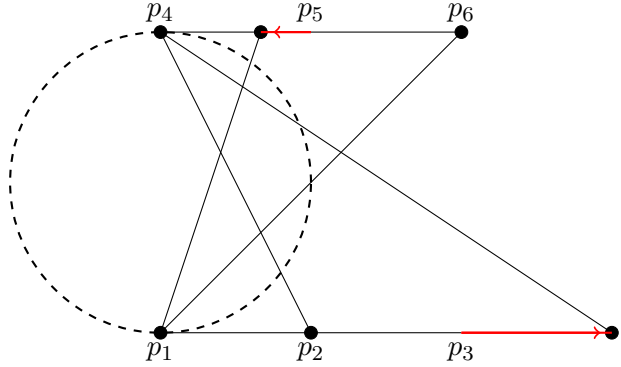

Finally, we revisit  the configuration in Figure \ref{autopolar} and analyze it using the orbit rigidity matrix.

The polarity $\pi_C$ maps the point $p_1$, with coordinates $(0,-1)$, to the line with homogeneous coordinates $(0:1:1)$, which is the tangent line to the circle at the point $p_1$. Call this line $L_1$. Similarly the point $p_4$ with coordinates $(0,1)$ is mapped by $\pi_C$ to the line $(0:-1:1)$, which is the tangent line to the circle at the point $p_4$. Call this line $L_4$.
The point $p_2$ with coordinates $(1,-1)$ is mapped by  to the line with homogeneous coordinates $(-1:1:1)$, which is the line through $p_1$ and $p_6$. Call this line $L_2$. Similarly, the point $p_5$, with coordinates $(1,1)$, is polar to the line through $p_3$ and $p_4$, which has homogeneous coordinates $(-1:-1:1)$. Call this line $L_5$.
The point $p_3$, with coordinates $(2,-1)$, is polar to the line through $p_2$ and $p_4$, with homogeneous coordinates $(-2:1:1)$. The point $p_6$, with coordinates $(2,1)$ is polar to the line through $p_1$ and $p_5$, with homogeneous coordinates $(-2:-1:1)$. Call these lines $L_3$ and $L_6$ respectively.

Under the polarity $\pi_C$, there are six orbits $\{p_i, L_i\}$ of points and lines. There are eight orbits of incidences: $i_0=\{(p_1,L_1)\}$,  $i_1=\{(p_1,L_2), (p_2, L_1)\}$, $i_2=\{(p_1,L_3), (p_3, L_1)\}$, $i_3=\{(p_4, L_4)\}$, $i_4=\{(p_4,L_5), (p_5, L_4)\}$,$i_5=\{(p_4, L_6),$\\$ (p_6, L_4)\}$, $i_6=\{(p_2, L_6), (p_6, L_2)\}$ and $i_7=\{(p_3, L_5), (p_5, L_3)\}$.
The orbit rigidity matrix $M_2^{\langle \pi_C \rangle}(S, \mathbf{p})$ is
\medskip

\noindent
\footnotesize
\[
\begin{array}{ccccccccccccccc}
     & &\{p_1, L_1\}&  & \{p_2, L_2\} & & \{p_3,L_3\} & &\{p_4, L_4\} & & \{p_5, L_5\}& & \{p_6, L_6\} \\
   i_0 & \ldelim({8}{0.5em} & 0 &2 &0 &0 &0 &0 &0 &0 &0 &0  &0 &0 &  \rdelim){8}{0.5em} \\
   i_1 && -1 & 1 & 0 & 1 &0 &0 &0 &0 &0 &0 & 0 & 0 &\\
   i_2 && -2 & 1 & 0 & 0& 0 & 1 & 0 & 0 & 0 & 0 & 0 & 0 \\
   i_3 && 0 & 0 & 0 & 0 & 0 &0 & 0 & -2 & 0 & 0 & 0 & 0 \\
   i_4 && 0 &0 & 0 & 0 & 0 & 0 & -1 & -1 & 0 & -1 & 0 & 0 \\
   i_5 && 0 & 0 & 0 & 0 & 0 & 0 & -2 & -1 & 0 & 0 & 0 & -1 \\
   i_6 && 0 & 0 & -2 & -1 & 0 & 0 & 0 & 0 & 0 & 0 & -1 & 1 \\
   i_7 && 0 & 0 & 0 & 0 & -1 & -1 & 0 & 0 & -2 & -1 & 0 & 0
\end{array}
\]
\normalsize

\medskip

This matrix has a four-dimensional kernel, which implies that there is a 4-dimensional space of infinitesimal motions that preserve autopolarity. One of these motions is non-trivial and extends to a  finite motion.

To see this, pin the points $p_1$, $p_2$, $p_4$ and $p_6$. The points $p_1$, $p_2$ and $p_3$ have to remain collinear after any projective motion, so the point $p_3$ can only move along the line $L_1$. Suppose that the point $p_3$ is moved to the point $p_3'=(2+t, -1)$. The line $L_5$ then has to be moved to the line between $p_3'$ and $p_4$ in order for the incidences to be preserved. Similarly, the point $p_5$ can only move along the line $L_4$. Suppose that the point $P_5$ is moved to $p_5'=(1+t',1)$. Choosing $t'=\frac{-t}{2+t}$ preserves autopolarity. Furthermore, $t'=\frac{-t}{2+t}$ is the only possible choice for $t'$ preserving autopolarity.

Hence the only non-trivial projective motion preserving autopolarity is moving $p_3$ to $(2+t,-1)$ and $p_5$ to $(1+\frac{-t}{2+t},1)$, and moving the lines $L_3$ and $L_5$ to preserve incidences. Figure \ref{autopolar_motion} illustrates this motion.

\section{Conclusion and future work}\label{sec:con}
Finally, we mention some possible future research directions. Some of these open problems will be elaborated on in the companion paper \cite{fields}.

\textit{Implications of self-stresses for realizability of incidence geometries.} What are the implications of dependencies among the incidences of an incidence geometry for whether or not the incidence geometry is realizable as points and straight lines in the projective plane? Does non-realizability of an incidence geometry imply that the incidences are dependent, or are there non-realizable incidence geometries with sets of incidences that are independent in the projective rigidity matroid? Note that points or lines coinciding does not necessarily imply that the incidences are dependent.

\textit{Inductive constructions.} 
There are well-known  constructions, known as $0$- and $1$-extensions, that preserve rigidity of bar-joint frameworks in $\mathbb{R}^d$. In the plane, $0$- and $1$-extensions are sufficient for inductively constructing all minimally rigid graphs, starting from a single edge. 

Operations in the Cayley algebra, i.e. adding a line between two points not already connected by a line in the configuration, or adding the intersection point of two lines whose point of intersection is not already in the configuration, preserve independence in the projective rigidity matrix, and can therefore be used to inductively construct projectively rigid configurations \cite{fields}. However, operations in the Cayley algebra are clearly not sufficient for constructing all projectively rigid configurations, so it would be interesting to find more inductive constructions that preserve projective rigidity.

\textit{Generalizations to higher dimension.} Incidence geometries can also be realized in real projective spaces of higher dimension. It would be natural to consider the generalization of projective motions to realizations of incidence geometries and points and hyperplanes in higher-dimensional real projective spaces. Another natural generalization would be to consider realizations of incidence geometries as points and lines in projective spaces of dimension three or higher. Investigating generalizations of projective motions to such realizations is a possible avenue for future research. 

\textit{Families of rigid $v_k$-configurations.} In Example~\ref{Rigid_ex}, we show a projectively rigid $20_4$-configuration. As $v_k$-configurations are overconstrained with respect to the count in \eqref{count_indep} whenever $k \geq 4$, it seems likely that there are more projectively rigid $v_k$-configurations. Constructing families of projectively rigid $v_k$-configurations, or proving that known families of $v_k$-configurations are projectively rigid, is another potential direction for future research.

\section*{Acknowledgements}

Most of our results presented here were obtained during the Focus Program on Geometric Constraint Systems
July 1 - August 31, 2023 and we gratefully acknowledge the productive work environment provided by the Fields Institute for Research in Mathematical Sciences, as well as its financial support. The contents of this article are solely the responsibility of the authors and do not necessarily represent the official views of the Institute.

The work was also supported by the Knut and Alice Wallenberg Foundation Grant 2020.0001 and 2020.0007.
\bibliographystyle{plain}

\bibliography{projective}
\end{document}